\documentclass[reqno]{amsart}

\usepackage{amssymb,amsmath,amsxtra,graphicx}
\usepackage{times}
\usepackage{bbold}
\usepackage{stmaryrd}
\usepackage{array}
\usepackage{mathrsfs}
\usepackage{dsfont}
\usepackage{url}
\usepackage{tikz}
\usetikzlibrary{shapes,fit,arrows,calc,backgrounds,positioning,matrix}
\usepackage[T1]{fontenc}

\theoremstyle{plain}
\newtheorem{thm}{Theorem}
\newtheorem{lem}[thm]{Lemma}
\newtheorem{cor}[thm]{Corollary}
\newtheorem{prop}[thm]{Proposition}

\theoremstyle{definition}

\newtheorem*{example}{Example}

\newtheorem*{ack}{Acknowledgment}
\newtheorem*{notat}{Notations and conventions}

\theoremstyle{remark}

\newcommand{\ZZ}{\mathbb{Z}}
\newcommand{\QQ}{\mathbb{Q}}
\newcommand{\sR}{\mathscr{R}}
\newcommand{\cO}{\mathcal{O}}

\newcommand{\ch}{\operatorname{char}}
\newcommand{\onto}{\twoheadrightarrow}

\newcommand{\into}{\hookrightarrow}

\newcommand{\longiso}{\overset{\text{\raisebox{-.5ex}{$\sim$}}}{\longrightarrow}}

\newcommand{\define}{\ \stackrel{\text{\rm def}}{=}\ }
\newcommand{\cprid}[1]{\boldsymbol{e}({#1})}

\newcommand{\sfrac}[2]{\text{\small $\frac{#1}{#2}$}}

\renewcommand{\k}{\mathbb{k}}
\newcommand{\gen}[1]{\langle{#1}\rangle}

\renewcommand{\bar}[1]{\overline{#1}}

\newcommand{\reg}{\text{\rm reg}}

\newcommand{\Aut}{\operatorname{Aut}}
\newcommand{\End}{\operatorname{End}}
\newcommand{\Hom}{\operatorname{Hom}}
\renewcommand{\Im}{\operatorname{Im}}

\DeclareMathOperator{\Mat}{Mat}

\newcommand{\trace}{\operatorname{trace}}
\newcommand{\Id}{\operatorname{Id}}

\newcommand{\cen}{\mathscr{Z}}

\newcommand{\Th}{^\text{\rm th}}
\renewcommand{\d}{\delta}

\newcommand{\ant}{\mathcal{S}}
\newcommand{\e}{\varepsilon}
\newcommand{\Hint}{\textstyle\int}

\DeclareMathOperator{\Ind}{Ind}

\DeclareMathOperator{\Irr}{Irr}

\newcommand{\rhk}{\text{\scriptsize $\rightharpoonup$}}
\newcommand{\lhk}{\text{\scriptsize $\leftharpoonup$}}

\newcommand{\Cas}[1]{c_{#1}}
\newcommand{\Hig}[1]{\gamma_{#1}}

\newcommand{\cat}[1]{\operatorname{\mathsf{#1}}}

\newcommand{\kalg}{\cat{Alg}_{\k}}

\newcommand{\Rep}[1]{\cat{rep}{#1}}

\newcommand{\Units}[1]{{#1}^{\boldsymbol \times}}
\newcommand{\bdot}{\,\text{\raisebox{-.45ex}{$\boldsymbol{\cdot}$}}\,}

\newcommand{\upin}{\text{\ \rotatebox{90}{$\in$}\ }}

\renewcommand{\emph}[1]{{\bfseries\itshape #1}}

\renewcommand{\labelenumi}{(\alph{enumi})}

\begin{document}




\title[Frobenius Divisibility for Hopf Algebras]%
{Frobenius Divisibility for Hopf Algebras}

\author{Adam Jacoby}

\author{Martin Lorenz}

\address{Department of Mathematics, Temple University,
    Philadelphia, PA 19122}


\dedicatory{To Don Passman on the occasion of his $75^\text{th}$
birthday}

\thanks{Research supported in part by NSA Grant H98230-15-1-0036}



\begin{abstract}
We present a unified ring theoretic approach, based on properties of the Casimir element of a symmetric
algebra, to a variety of known divisibility results for the degrees of irreducible representations
of semisimple Hopf algebras in characteristic $0$. 
All these results are motivated by a classical theorem of 
Frobenius on the degrees of irreducible complex representations
of finite groups.
\end{abstract}

\maketitle



\section*{Introduction}

\subsection{} 
A result of Frobenius from 1896 \cite[\S 12]{fF96} states that the degree of every irreducible
representation of a finite group $G$ over an algebraically closed field $\k$
characteristic $0$ divides the order of $G$. According to a conjecture
of Kaplansky, the $6\Th$ in his celebrated list of conjectures in \cite{iK75}, an analogous fact is expected to
hold more generally for any finite-dimensional involutory Hopf $\k$-algebra $H$: 
\smallskip

\centerline{\textbf{FD}:\quad the degrees of all
irreducible representations of $H$ divide $\dim_\k H$.}
\smallskip

\noindent
Recall that $H$ is said to be involutory if the antipode 
of $H$ satisfies $\ant^2 = \Id_H$ -- it is a standard fact that, over a base field $\k$ of
characteristic $0$, this condition amounts to semisimplicity of $H$. 
Frobenius' original theorem is the special case of FD where $H$ is the group algebra
of $G$ over $\k$; therefore, the statement FD is referred to as \emph{Frobenius divisibility}. While 
Kaplansky's conjecture remains open as of this writing, it is in fact known to hold in several instances.
It is our aim in this article to present a unified approach
to a number of these generalizations using some general observations on symmetric algebras.
Some of this material can be found in the earlier article \cite{mL11} 
in the slightly more general setting of Frobenius algebras, and the authors have also
greatly benefitted from a reading of the preprint \cite{jCeMxx} by Cuadra and Meir.

\subsection{} 
Here is a brief overview of the contents of this article. Section~\ref{S:Symmetric} is entirely devoted to
symmetric algebras, with particular focus on the special case of a finite-dimensional semisimple algebra $A$.
The main results, Proposition~\ref{P:cprid} and Theorem~\ref{T:Casimir}, both concern the so-called \emph{Casimir element}
of $A$; the former gives an expression of the central primitive idempotents of $A$ in terms of the Casimir element
while the latter gives a representation theoretic
description of the Casimir element. Under certain hypotheses, these results
allow us to deduce that a version of property FD for $A$ is in fact equivalent to the apparently 
simpler condition that the Casimir element is
integral over $\ZZ$; this is spelled out in Corollary~\ref{C:Divisibility}. 
As a first application, we quickly derive Frobenius' original divisibility 
theorem for finite group algebras in \S\ref{SS:FrobFinitekG}. 

Section~\ref{S:Hopf} then concentrates on a semisimple Hopf algebra $H$ over a field $\k$ of characteristic $0$.
We start with a formulation of Corollary~\ref{C:Divisibility}, due to Cuadra and Meir \cite{jCeMxx}, that is specific to $H$;
the result (Theorem~\ref{T:FrobHopf}) states that FD for $H$ is equivalent to integrality (over $\ZZ$)
of the Casimir element $\Cas\lambda$ of $H$ that is associated to the 
unique integral of $\lambda \in H^*$ satisfying $\gen{\lambda,1} = 1$. Further results derived in Section~\ref{S:Hopf}
from the material on symmetric algebras are: a theorem of Zhu \cite{sZ93} on irreducible representations of $H$
whose character is central in $H^*$, the so-called class equation \cite{gK72}, \cite{yZ94}, and a theorem of 
Schneider \cite{hjS01} on factorizable Hopf algebras. All these results ultimately are consequences of certain
ring theoretic properties of the subalgebra $\sR_\k(H)$ of $H^*$ that is spanned by the characters of representations of $H$:
the algebra $\sR_\k(H)$ is semsimple and defined over $\ZZ$, that is, 
$\sR_\k(H) \cong \k \otimes_{\ZZ} \sR(H)$ for some subring
$\sR(H)$. Moreover, $\sR(H)$ is finitely generated over $\ZZ$ and the Casimir element of  $\sR_\k(H)$
does in fact belong to  $\sR(H)^{\otimes 2}$, thereby ensuring that the requisite integrality property holds.
We emphasize that none of the results in Section~\ref{S:Hopf} are new; 
we take credit only for the presentation and the unified 
approach. However, we hope that the methods of this article and the point of view promulgated here
will prove useful in making further progress
toward a resolution of Kaplansky's conjecture.

\subsection{} 
With gratitude and admiration, the authors dedicate this article to Don Passman.
The senior author has benefitted throughout his career
from Don's mathematical insights, his generosity in sharing ideas and his long lasting support and friendship.
This paper bears witness to the fact that Don has profoundly influenced 
generations of algebraists. 

\begin{notat}
Throughout, we work over a base field $\k$ 
and $\otimes = \otimes_\k$\,. For any $\k$-vector space $V$, we let 
$\gen{\bdot , \bdot} \colon V^* \times V \to \k$
denote the evaluation pairing.
The center of a $\k$-algebra $A$ is  denoted by $\cen A$ and 
the unit group by $\Units A$.
Furthermore, $\Irr A$ will denote a full representative set of the isomorphism classes of 
irreducible representations of $A$ and $\Rep A$ is the category of finite-dimensional representations of $A$ or, equivalently, 
finite-dimensional left $A$-modules.
Our notation and terminology concerning Hopf algebras is standard and follows Montgomery \cite{sM93}
and Radford \cite{dR12}.
\end{notat}


\section{Symmetric Algebras} 
\label{S:Symmetric}

\subsection{Symmetric Algebras and Frobenius Forms}
\label{SS:SymmDef}

Every $\k$-algebra $A$ carries the ``regular'' $(A,A)$-bimodule structure: the left action of $a \in A$ 
on $A$ is given by the left multiplication operator, $a_A$, and the right
action by right multiplication, ${}_Aa$.
This structure gives rise to a bimodule structure on the linear dual $A^* = \Hom_\k(A,\k)$
for which the following notation is customary in the Hopf literature:
\begin{equation*} 
a \rhk f \lhk b \define f \circ b_A \circ {}_Aa 
\qquad \text{or} \qquad
\gen{a \rhk f \lhk b , c} = \gen{f, bca}
\end{equation*}
for $a,b,c \in A$ and $f \in A^*$. 
The algebra $A$ is said to be \emph{symmetric} if $A\cong A^*$ as $(A,A)$-bimodules.
Note that even a mere $\k$-linear isomorphism $A^* \cong A$ forces $A$ to be finite-dimensional;
so symmetric algebras will necessarily have to be finite-dimensional. 
\smallskip

The image of $1\in A$ under any $(A,A)$-bimodule isomorphism $A \longiso A^*$ is a linear
form $\lambda \in A^*$ satisfying $a \rhk \lambda = \lambda \lhk a$ for all $a\in A$ or,
equivalently, $\gen{\lambda,ab} = \gen{\lambda,ba}$ for all $a,b\in A$. Linear forms 
satisfying this condition will be called \emph{trace forms}. 
Moreover, $a \rhk \lambda = 0$ forces $a=0$, which is equivalent to the
fact that $\lambda$ does not vanish on any nonzero ideal of $A$; this condition will be referred to as 
\emph{nondegeneracy}. Conversely, if $A$ is a 
finite-dimensional $\k$-algebra that is equipped with a nondegenerate trace form $\lambda \in A^*$, 
then we obtain an $(A,A)$-bimodule
isomorphism
\begin{equation}
\label{E:Iso}
\begin{tikzpicture}[baseline=(current  bounding  box.center),  >=latex, scale=.7,
bij/.style={above,sloped,inner sep=0.5pt}]
\matrix (m) [matrix of math nodes, row sep=.1em,
column sep=2.5em, text height=1.5ex, text depth=0.25ex]
{A & A^*  \\ 
\upin & \upin \\ 
a &   a \rhk \lambda = \lambda \lhk a\\}; 
\draw[->] 
(m-1-1) edge node[bij] {$\sim$}  (m-1-2);
\draw[|->] 
(m-3-1) edge (m-3-2);
\end{tikzpicture} 
\end{equation}
Thus, $A$ is symmetric. We will refer to $\lambda$ as a \emph{Frobenius form} for $A$. Note that
$\lambda$ is determined up to multiplication by a central unit of $A$: the possible Frobenius
forms of $A$ are given by 
\[
\lambda' = u \rhk \lambda \qquad\text{with}\qquad u \in \cen A \cap \Units A
\]
We will think of a symmetric algebra as a pair
$(A,\lambda)$ consisting of the algebra $A$ together with a fixed Frobenius trace form $\lambda$.
A \emph{homomorphism} $f \colon (A,\lambda) \to (B,\mu)$ of symmetric
algebras is a $\k$-algebra map $f\colon A \to B$ such that $\lambda = \mu \circ f$.

\subsection{First Examples of Symmetric Algebras}
\label{SS:FirstExamples}

To put the above definitions into perspective, we offer the following examples of symmetric algebras. 
Another important example will occur later in this article (\S\ref{SS:ClassEqn}).

\subsubsection{Finite Group Algebras}
\label{SSS:kG}
The group algebra $\k G$ of any finite group $G$ is symmetric. Indeed, any
$a \in \k G$ has the form $a = \sum_{g\in G} \alpha_g g$ with unique scalars $\alpha_g \in \k$
and a Frobenius form for $\k G$ is given by
\[
\lambda(\sum_{g\in G} \alpha_g g) \define \alpha_1
\]
It is straightforward to check that $\lambda$ is indeed a non-degenerate trace form.

\subsubsection{Finite-dimensional Semisimple Algebras}
\label{SSS:SS}
Any finite-dimensional semisimple algebra $A$ is symmetric. In order to obtain a Frobenius form
for $A$, it suffices to construct a nonzero trace form for all simple Wedderburn
components of $A$; the sum of these trace forms will then be a Frobenius form for $A$.
Thus, we may assume that $A$ is a
finite-dimensional simple $\k$-algebra. Letting $\bar{K}$ denote an
algebraic closure of the $\k$-field $K= \cen A$, we have $A \otimes_K\bar{K}\cong \Mat_d(\bar{K})$
for some $d$. The ordinary matrix trace $\Mat_d(\bar{K}) \to \bar{K}$
yields a nonzero $\k$-linear map
$A \into A \otimes_K\bar{K} \to \bar{K}$
that vanishes on the space $[A,A]$ of all Lie commutators in $A$. Thus, $A/[A,A] \neq 0$
and we may pick any $0 \neq \lambda_0 \in (A/[A,A])^*$ to obtain a
nonzero trace form $\lambda = \lambda_0 \circ \text{can} \colon A \onto A/[A,A] \to \k$. Non-degeneracy of
$\lambda$ is clear, because $A$ has no nonzero ideals.

\subsection{The Casimir Element}
\label{SS:Casimir}

Let $(A,\lambda)$ be a symmetric algebra. The element $c_\lambda \in A\otimes A$ that corresponds
to $\Id_A \in \End_\k(A)$ under the 
isomorphism $\End_\k(A) \cong A\otimes A$ coming from \eqref{E:Iso} is called the
\emph{Casimir element} of $(A,\lambda)$:
\begin{equation}
\label{E:Casimir}
\begin{tikzpicture}[baseline=(current  bounding  box.center),  >=latex, scale=.7,
bij/.style={above,sloped,inner sep=0.5pt}]
\matrix (m) [matrix of math nodes, row sep=.1em,
column sep=2.5em, text height=1.5ex, text depth=0.25ex]
{\End_\k(A) & A \otimes A^* & A \otimes A \\ 
\upin & & \upin \\ 
\Id_A &  & \Cas{\lambda}  \\}; 
\draw[->] 
(m-1-1) edge node[bij] {$\sim$} node[below] {\scriptsize can.} (m-1-2)
(m-1-2) edge node[bij] {$\sim$} node[below] {\scriptsize \eqref{E:Iso}} (m-1-3);
\draw[|->] 
(m-3-1) edge (m-3-3);
\end{tikzpicture} 
\end{equation}
Writing $\Cas{\lambda} = \sum_i x_i \otimes y_i$ with $\{ x_i\}$ $\k$-linearly independent, \eqref{E:Casimir} means 
explicitly that $a = \sum_i x_i \gen{\lambda, ay_i}$ for all $a \in A$ or, equivalently,
the $x_i$ form a $\k$-basis of $A$ 
such that $\gen{\lambda,x_iy_j} = \d_{i,j}$ for all $i,j$. It then follows that the $y_i$ also form a
$\k$-basis of $A$ satisfying $\gen{\lambda,y_ix_j} = \d_{i,j}$. Hence,
\begin{equation}
\label{E:CasimirSymm}
\Cas{\lambda} = \sum_i x_i \otimes y_i = \sum y_i \otimes x_i
\end{equation}
Thus, the Casimir element $\Cas{\lambda}$ is fixed by the switch map 
$\tau \in \Aut_{\kalg}(A \otimes A)$ given by $\tau(a\otimes b) = b \otimes a$.

\begin{lem}
\label{L:Casimir}
Let $(A,\lambda)$ be a symmetric algebra. Then, for all $z \in A \otimes A$, we have 
$z \Cas{\lambda} \ = \ \Cas{\lambda} \tau(z)$. 
Consequently, $\Cas{\lambda}^2 \in \cen(A\otimes A) = \cen A \otimes \cen A$\,.
\end{lem}

\begin{proof}
Recall that $a = \sum_i x_i \gen{\lambda, ay_i} = \sum_i y_i \gen{\lambda, ax_i}$ for all $a \in A$. Using this, we
compute
\[
\sum_i ax_i \otimes y_i =
\sum_{i,j} x_j\gen{\lambda, ax_iy_j} \otimes y_i =
\sum_{i,j} x_j  \otimes y_i \gen{\lambda, y_j ax_i} =
\sum_j x_j \otimes y_j a
\]
Thus, $(a \otimes 1)\Cas\lambda = \Cas\lambda(1\otimes a)$. Applying the switch automorphism 
$\tau$ to this equation
and using the fact that $\Cas\lambda$ is stable under $\tau$, we also obtain
$(1 \otimes b)\Cas\lambda = \Cas\lambda(b \otimes 1)$. Hence, 
$(a \otimes b) \Cas{\lambda} \ = \ \Cas{\lambda}(b \otimes a)$,
which implies $\Cas{\lambda}^2 \in \cen(A\otimes A)$.
\end{proof}

\subsection{The Casimir Trace}
\label{SS:CasimirTr}

The following operator was originally introduced by D.G.~Higman \cite{dH55}:
\begin{equation}
\label{E:CasTr}
\begin{tikzpicture}[baseline=(current  bounding  box.center),  >=latex, scale=.7]
\matrix (m) [matrix of math nodes, row sep=.1em,
column sep=2em, text height=1.5ex, text depth=0.25ex]
{\Hig{\lambda} \colon \quad A &  A/[A,A] & \cen A   \\ \hspace{.25in} \quad \upin &&  \upin \\ 
\hspace{.4in} a  &   & \sum_i x_i a y_i = \sum_i y_i a x_i\\};
\draw[->>] (m-1-1) edge node[auto] {\scriptsize can.} (m-1-2);
\draw[->] (m-1-2) edge (m-1-3);
\draw[|->] (m-3-1) edge (m-3-3);
\end{tikzpicture}
\end{equation}
The following lemma justifies the claims, implicit in \eqref{E:CasTr}, that  
$\Hig\lambda$ is center-valued and vanishes on $[A,A]$.
We will refer to $\Hig{\lambda}$ as the \emph{Casimir trace} of $(A,\lambda)$.

\begin{lem}
\label{L:HigCas}
Let $(A,\lambda)$ be a symmetric algebra. Then
$a\Hig{\lambda}(bc) = \Hig{\lambda}(cb)a$ for all $a,b,c \in A$.
\end{lem}

\begin{proof}
The identity in Lemma~\ref{L:Casimir} states that
$\sum_i ax_i \otimes by_i = \sum_i x_i b \otimes y_i a$ in $A \otimes A$.
Multiplying this identity on the right with $c\otimes 1$
and then applying the multiplication map $A \otimes A \to A$ gives
$\sum_i ax_ic by_i = \sum_i x_i bc y_i a$ or, equivalently,
$a\Hig{\lambda}(cb) = \Hig{\lambda}(bc)a$ as claimed.
\end{proof}

\subsection{A Trace Formula}
\label{SS:Trace}

The Casimir element $\Cas\lambda$ can be used to give a convenient trace formula for endomorphisms of $A$:

\begin{lem}
\label{L:FrobTraces}
Let $(A,\lambda)$ be a symmetric algebra with Casimir element $\Cas\lambda = \sum_i x_i \otimes y_i$.
Then, for any $f \in \End_\k(A)$, we have
$\trace(f) = \sum_i \gen{\lambda, f(x_i )y_i} = \sum_i \gen{\lambda, x_i f(y_i)}$.
\end{lem}

\begin{proof}
By \eqref{E:Casimir}, $f(a) = \sum_i f(x_i) \gen{\lambda, ay_i}$ for all $a \in A$. Thus, 
\begin{equation*}
\begin{tikzpicture}[baseline=(current  bounding  box.center),  >=latex, scale=.7,
bij/.style={above,sloped,inner sep=0.5pt}]
\matrix (m) [matrix of math nodes, row sep=.1em,
column sep=2.2em, text height=1.5ex, text depth=0.25ex]
{\trace \colon &[-2em] \End_\k(A) & A \otimes A^*  & \k \\ 
& \upin & \upin & \upin \\ 
& f &  \sum_i f(x_i) \otimes (y_i \rhk \lambda) & \sum_i \gen{\lambda, f(x_i )y_i} \\}; 
\draw[->] 
(m-1-2) edge node[bij] {$\sim$} node[below] {\scriptsize can.} (m-1-3)
(m-1-3) edge node[below] {\scriptsize evaluation} (m-1-4);
\draw[|->] 
(m-3-2) edge (m-3-3)
(m-3-3) edge (m-3-4);
\end{tikzpicture} 
\end{equation*}
This proves the first equality; the second follows from \eqref{E:CasimirSymm}.
\end{proof}

With $f = b_A \circ {}_Aa$ for $a,b \in A$,  Lemma~\ref{L:FrobTraces} gives the formula
\begin{equation*}
\trace(b_A \circ {}_Aa) = \langle \lambda, b \Hig{\lambda}(a) \rangle 
= \langle \lambda, \Hig{\lambda}(b)a \rangle = \trace(a_A \circ {}_Ab)
\end{equation*}
In particular, we obtain the following expressions for the
\emph{regular character} of $A$: 
\begin{equation}
\label{E:Traces.2}
\gen{\chi_\reg,a} \define \trace(a_A)  = \trace({}_Aa) \\
= \gen{\lambda, \Hig{\lambda}(a)}
= \gen{\lambda, \Hig{\lambda}(1)a} 
\end{equation}

\subsection{Primitive Central Idempotents}
\label{SS:CentralPrimitive}

Now let $A$ be a finite-dimensional semisimple $\k$-algebra and let 
$\Irr A$ denote a full representative set of the isomorphism classes of 
irreducible representations of $A$. For each $S \in \Irr A$, we let $D(S) = \End_A(S)$ denote the 
Schur division algebra of $S$ and $a_S \in \End_{D(S)}(S)$ the operator given by the action of $a$ on $S$.
Consider the Wedderburn isomorphism
\begin{equation}
\label{E:Wedd}
\begin{tikzpicture}[baseline=(current  bounding  box.center),  >=latex, scale=.7,
bij/.style={above,sloped,inner sep=0.5pt}, text height=1.5ex, text depth=0.25ex]
\node(11) at (-2.3,-.2){$a$};
\node(12) at (1.1,-.2){$\big( a_S \big)$};
\node(21) at (-2.3,.8){$\upin$};
\node(22) at (1.1,.8){$\upin$};
\node(31) at (-2.3,2){$A$};
\node(32) at (1.1,2){$\displaystyle \prod_{S \in \Irr A} \End_{D(S)}(S)$};
\draw[->] (31) edge  node[bij] {$\sim$}  (32);
\draw[|->] (11) edge (12);
\end{tikzpicture}
\end{equation}
The \emph{primitive central idempotent} $\cprid{S} \in \cen A$ 
is the element corresponding to 
$(0,\dots,0,\Id_S,0,\dots,0) \in \prod_{S \in \Irr A} \End_{D(S)}(S)$ under the this isomorphism; so
\[
\cprid{S}_{T} = \delta_{S,T} \Id_S \qquad (S,T \in \Irr A)
\]
The following proposition gives a formula for $\cprid{S}$ using 
data coming from the structure of $A$
as a symmetric algebra (\S\ref{SSS:SS}) and the \emph{character} $\chi_S$ of $S$, defined by
\[
\gen{\chi_S,a} = \trace( a_S ) \qquad (a \in A)
\]

\begin{prop}
\label{P:cprid}
Let $A$ be a finite-dimensional semisimple $\k$-algebra with Frobenius trace form $\lambda$\,.
Then, for each $S \in \Irr A$\,, we have the following formula in $A = \k \otimes A$:
\[
\Hig{\lambda}(1)\, \cprid{S} = d(S)\, (\chi_S \otimes \Id_A)(\Cas{\lambda})
= d(S)\, (\Id_A \otimes \chi_S)(\Cas{\lambda})
\]
where $d(S) = \dim_{D(S)}S$. In particular, $\Hig{\lambda}(1)_S = 0$ if and only if $\chi_S = 0$ 
or $d(S) \,1_\k = 0$.
\end{prop}

\begin{proof}
Since $(\chi_S \otimes \Id_A)(\Cas{\lambda})
= (\Id_A \otimes \chi_S)(\Cas{\lambda})$ by \eqref{E:CasimirSymm}, we only need to show that
$\Hig{\lambda}(1)\,\cprid{S} = d(S)\, (\chi_S \otimes \Id_A)(\Cas{\lambda})$. This
amounts to the condition $\gen{\lambda, \Hig{\lambda}(1)\,\cprid{S}\,a} 
= d(S)\, \gen{\lambda, (\chi_S \otimes \Id_A)(\Cas{\lambda})a}$ for all $a \in A$
by nondegeneracy of $\lambda$. But $a = \sum_i x_i\gen{\lambda,y_ia}$ by \eqref{E:CasimirSymm} and so
\[
\gen{\lambda, (\chi_S \otimes \Id_A)(\Cas{\lambda})a} 
\underset{\eqref{E:CasimirSymm}}{=} \gen{\lambda, {\textstyle \sum_i}\, \gen{\chi_S,x_i} y_i a}
= {\textstyle \sum_i}\, \gen{\chi_S,x_i}\gen{\lambda,y_ia} = \gen{\chi_S,a}
\]
Thus, our goal is to show that
\begin{equation}
\label{E:cpridSymm.0}
\gen{\lambda,\Hig{\lambda}(1)\,\cprid{S}\,a}  = d(S)\, \gen{\chi_S,a} \qquad (a \in A)
\end{equation}
For this, we use the regular character:
\[
\gen{\chi_\reg,\cprid{S}\,a} \underset{\eqref{E:Traces.2}}{=} 
\gen{\lambda,\Hig{\lambda}(1)\, \cprid{S}\,a} 
\]
On the other hand, by Wedderburn's Structure Theorem, the regular representation of $A$ has the form
$A_\reg \cong \bigoplus_{T \in \Irr A} T^{\oplus d(T)}$,
whence $\chi_\reg = \sum_{T \in \Irr A} d(T)  \chi_{T}$. 
Since  $\cprid{S} \rhk \chi_{T} = \chi_T \lhk \cprid{S} = \delta_{S,T}\chi_S$, we obtain
\begin{equation}
\label{E:cpridSymm}
\cprid{S} \rhk \chi_\reg  = \chi_\reg \lhk \cprid{S} = d(S) \chi_S
\end{equation}
Therefore, $\gen{\chi_\reg, \cprid{S} a} = d(S) \gen{\chi_S,a}$, proving \eqref{E:cpridSymm.0}.
Finally, \eqref{E:cpridSymm.0} also shows that $\Hig{\lambda}(1)\,\cprid{S} = 0$ if and only if
$d(S) \chi_S = 0$, which implies the last assertion in the proposition. 
\end{proof}

\subsection{The Casimir Square}
\label{SS:CasimirSquare}

Continuing to assume that $A$ be a finite-dimensional semisimple $\k$-algebra, we now describe the 
Casimir square $\Cas{\lambda}^2 \in \cen A \otimes \cen A$ (Lemma~\ref{L:Casimir}) 
in terms of the following isomorphism coming from the
Wedderburn isomorphism \eqref{E:Wedd}:
\begin{equation}
\label{E:A2Iso}
\begin{tikzpicture}[baseline=(current  bounding  box.center),  >=latex, scale=.7,
bij/.style={above,sloped,inner sep=0.5pt}]
\matrix (m) [matrix of math nodes, row sep=.6em,
column sep=2em, text height=1.5ex, text depth=0.25ex]
{A \otimes A & {\displaystyle \prod_{S,T \in \Irr A} \End_{D(S)}(S) \otimes \End_{D(T)}(T)} \\[3pt] 
\upin & \upin \\ a \otimes b & \big( a_S \otimes b_T \big) \\};
\draw[->] (m-1-1) edge node[bij] {$\sim$}  (m-1-2);
\draw[|->] (m-3-1) edge (m-3-2);
\end{tikzpicture} 
\end{equation}
We will write $t_{S,T} \in \End_{D(S)}(S) \otimes \End_{D(T)}(T)$ for the $(S,T)$-component of the
image of $t \in A\otimes A$ under the above isomorphism; so $(a \otimes b)_{S,T}= a_S \otimes b_T$. 
Recall that $S \in \Irr A$ is absolutely irreducible if and only if $D(S) = \k$.

\begin{thm}
\label{T:Casimir}
Let $A$ be a finite-dimensional semisimple $\k$-algebra with Frobenius trace form $\lambda$\,. Then
$(\Cas{\lambda})_{S,T} = 0$ for $S \neq	 T \in \Irr A$. If $S$ is absolutely irreducible, then
$(\dim_\k S)^2\,(\Cas{\lambda}^2)_{S,S} = \Hig{\lambda}(1)_S^2$.
\end{thm}

\begin{proof}
For $S \neq T$, we have
\[
(\Cas{\lambda})_{S,T} 
= \big( (\cprid{S} \otimes \cprid{T})\Cas{\lambda} \big)_{S,T} 
\underset{\text{Lemma~\ref{L:Casimir}}}{=} \big( \Cas{\lambda} (\cprid{T} \otimes \cprid{S}) \big)_{S,T} 
=  (\Cas{\lambda})_{S,T} (0_S \otimes 0_T) 
= 0
\]
It remains to consider $(\Cas{\lambda}^2)_{S,S}$. First, 
\begin{equation}
\label{E:CasSquare}
\Cas{\lambda}^2 = {\textstyle \sum_i}\, (x_i \otimes y_i)\Cas{\lambda}
\underset{\text{Lemma~\ref{L:Casimir}}}{=} 
{\textstyle \sum_i}\, (x_i \otimes 1)\Cas{\lambda}(y_i \otimes 1) = (\Hig{\lambda} \otimes \Id)(\Cas{\lambda})
\end{equation}
Next, for $c \in \cen A$, the operator $c_S \in D(S)$ is a scalar, since $S$ absolutely irreducible,
and $\chi_S(c) = d(S) c_S$ with $d(S) = \dim_\k S$. Therefore, writing $\rho_S(a) = a_S$ for $a \in A$, we calculate
\begin{equation}
\label{E:*}
\begin{aligned}
d(S) (\rho_S \circ \Hig{\lambda}) (a) &=
(\chi_S \circ \Hig{\lambda})(a)  = \chi_S({\textstyle \sum_i} x_i a y_i) = \chi_S({\textstyle \sum_i} a y_ix_i) \\
&= \chi_S(a \,\Hig{\lambda}(1)) = \chi_S(a)\, \Hig{\lambda}(1)_S
\end{aligned}
\end{equation}
and further 
\[
\begin{aligned}
d(S)^2\,(\Cas{\lambda}^2)_{S,S} 
&\hspace{1.27em} \underset{\eqref{E:CasSquare}}{=} \hspace{1.27em} 
d(S)^2\,(\rho_S \otimes \rho_{S})\big((\Hig{\lambda} \otimes \Id)(\Cas{\lambda})\big)\\
&\hspace{1.45em} = \hspace{1.45em} d(S)^2\,\big((\rho_S \circ \Hig{\lambda}) \otimes \rho_S\big)(\Cas{\lambda})\\
&\hspace{1.27em} \underset{\eqref{E:*}}{=} \hspace{1.27em} 
d(S)\,(\chi_S  \otimes \rho_S )(\Cas{\lambda})\, \Hig{\lambda}(1)_S\\
&\hspace{1.45em} = \hspace{1.45em} 
(\Id_\k \otimes \rho_S )\big( d(S)\,(\chi_S  \otimes \Id)(\Cas{\lambda})\big)\, \Hig{\lambda}(1)_S\\
&\underset{\text{Proposition~\ref{P:cprid}}}{=} 
\rho_S\big( \cprid{S}\Hig{\lambda}(1) \big)\, \Hig{\lambda}(1)_S = \Hig{\lambda}(1)_S^2 \\
\end{aligned}
\]
which completes the proof of the theorem.
\end{proof}

\subsection{Integrality and Divisibility}
\label{SS:Divisibility}

We recall some standard facts about integrality.
Let $R$ be a ring and let $S$ be a subring of the center 
$\cen R$. An element $r \in R$ is said to be \emph{integral}
over $S$ if $r$ satisfies some monic polynomial over $S$. 
The following basic facts will be used repeatedly below:
\begin{itemize}
\item
An element $r \in R$ is integral over $S$ if and only if $r \in R'$ for some subring $R'\subseteq R$
such that $R'$ contains $S$ and is finitely generated as an $S$-module.
\item
If $R$ is commutative, then the elements of $R$ that are integral over $S$ 
form a subring of $R$ containing $S$, called the
\emph{integral cosure} of $S$ in $R$.
\item
The integral closure of $\ZZ$ in $\QQ$ is $\ZZ$: an 
element of $\QQ$ that is integral over $\ZZ$ must belong to $\ZZ$.
\end{itemize}
The last fact above reduces the problem of
showing that a given nonzero integer $s$ divides another integer $t$ to proving that the 
fraction $\frac{t}{s}$ is merely integral over $\ZZ$. 

A semisimple $\k$-algebra $A$ is said to be \emph{split} if $D(S) = \k$ holds for all $S \in \Irr A$.

\begin{cor}
\label{C:Divisibility}
Let $A$ be a 
split semisimple $\k$-algebra with Frobenius trace form $\lambda$. Assume
that $\ch \k = 0$ and that $\Hig{\lambda}(1) \in \ZZ$\,.
Then the following are equivalent:
\begin{enumerate}
\renewcommand{\labelenumi}{\normalfont (\roman{enumi})}
\item
The degree of every irreducible representation of $A$ divides $\Hig{\lambda}(1)$;
\item
the Casimir element $\Cas{\lambda}$ is integral over $\ZZ$.
\end{enumerate}
\end{cor}

\begin{proof}
Theorem~\ref{T:Casimir} gives the formula
$(\Cas{\lambda}^2)_{S,S} =  \big( \frac{\Hig{\lambda}(1)}{\dim_\k S} \big)^2$.
If (i) holds, then the isomorphism \eqref{E:A2Iso} maps
$\ZZ[c_{\lambda}^2]$ into $\prod_{S \in \Irr H} \ZZ$, because $(c_{\lambda})_{S,T} = 0$ 
for $S \neq T$ by
Theorem~\ref{T:Casimir}. Thus, $\ZZ[c_{\lambda}]$ is a finitely generated $\ZZ$-module and (ii) follows.
Conversely, (ii) implies that
$c_{\lambda}^2$ also satisfies a monic polynomial over $\ZZ$ and all 
$(c_{\lambda}^2)_{S,S}$ satisfy the same
polynomial. Therefore, the fractions $\frac{\Hig{\lambda}(1)}{\dim_\k S}$ 
must be integers, proving (i). 
\end{proof}

Next, for a given homomorphism $(A,\lambda) \to (B,\mu)$ of symmetric algebras, we may 
consider the induced module $\Ind_A^B S = B \otimes_A S$ for each $S \in \Irr A$

\begin{cor}
\label{C:Divisibility.2}
Let $A$ be a 
split semisimple algebra over a field $\k$ of characteristic $0$ and let $\lambda$
be a Frobenius trace form for $A$.
Furthermore, let $(B,\mu)$ be a symmetric $\k$-algebra such that $\Hig{\mu}(1) \in \k$ and let
$(A,\lambda) \to (B,\mu)$ be a homomorphism of symmetric algebras.
If the Casimir element $\Cas{\lambda}$ is integral over $\ZZ$, then so is
the scalar $\frac{\Hig{\mu}(1)}{\dim_\k \Ind_A^B S}$ for each $S \in \Irr A$.
\end{cor}

\begin{proof}
It suffices to show that 
\begin{equation*}
\sfrac{\Hig{\mu}(1)}{\dim_\k \Ind_A^B S} = \sfrac{\Hig{\lambda}(1)_S}{\dim_\k S}
\end{equation*}
Indeed, by Theorem~\ref{T:Casimir}, the square of the fraction on the right equals 
$(\Cas{\lambda}^2)_{S,S}$\,, which is integral over $\ZZ$ if $\Cas{\lambda}$ is.
To check the above equality, let us put $e:= \cprid{S}$ for brevity. Then
$S^{\oplus \dim_\k S} \cong A e$ and so $\Ind_A^B S^{\oplus \dim_\k S} \cong B\phi(e)$, where 
$\phi$ denotes the given homomorphism $(A,\lambda) \to (B,\mu)$.
Since $\phi(e) \in B$ is an idempotent, $\dim_\k B\phi(e) = \trace({}_B\phi(e))$. Therefore,
\[
\begin{aligned}
\dim_\k \Ind_A^B S^{\oplus \dim_\k S} &= \trace({}_B\phi(e))
\underset{\eqref{E:Traces.2}}{=} \gen{\mu, \phi(e) \,\Hig{\mu}(1)} 
= \gen{\mu, \phi(e)} \,\Hig{\mu}(1) \\ 
&= \gen{\lambda, e} \,\Hig{\mu}(1) 
\underset{\eqref{E:cpridSymm.0}}{=}  \sfrac{(\dim_\k S)^2}{\Hig{\lambda}(1)_S} \Hig{\mu}(1)
\end{aligned}
\]
The desired equality is immediate from this.
\end{proof}

\subsection{A First Application: Frobenius' Divisibility Theorem for Finite Group Algebras}
\label{SS:FrobFinitekG}

Returning to the setting of \S\ref{SSS:kG}, consider
the group algebra $\k G$ of any finite group $G$ and assume that $\k$ is a splitting field for $\k G$ with $\ch \k = 0$;
so $\k G$ is split semisimple. 
The Frobenius form $\lambda$
of \S\ref{SSS:kG} satisfies $\gen{\lambda, gh^{-1}} = \d_{g,h}$ for $g,h\in G$. Hence, the Casimir element is 
\[
\Cas\lambda = \sum_{g\in G} g \otimes g^{-1}
\]
Since $\Cas\lambda \in \ZZ G \otimes_\ZZ \ZZ G$, a subring of $\k G \otimes \k G$ that is
finitely generated over $\ZZ$, condition (ii) in Corollary~\ref{C:Divisibility} is satisfied.
Moreover, $\Hig\lambda(1) = |G| \in \ZZ$ as also required in Corollary~\ref{C:Divisibility}. Thus, the corollary
yields that the degrees of all irreducible representations of $\k G$ divide $\Hig{\lambda}(1) = |G|$, 
as stated in Frobenius' classical theorem.


\section{Hopf Algebras} 
\label{S:Hopf}

\subsection{Preliminaries: Semisimplicity, Integrals, and Frobenius Forms}

We begin with a few reminders on semisimple Hopf algebras $H$ over a field $\k$
of characteristic $0$; the reader is referred to \cite{sM93},  \cite{dR12} and \cite{hjS95} for details. First, 
semisimplicity of $H$ amounts to $H$ being finite-dimensional and
involutory, that is, the antipode of $H$ satisfies $\ant^2 = \Id_H$. Both properties pass to $H^*$; so
$H^*$ is semisimple as well. 
By Maschke's Theorem for Hopf algebras, $H$ is unimodular and $\gen{\e,\Lambda} \neq 0$ holds for any
nonzero integral $\Lambda \in \Hint_H$, where $\e$ is the counit of $H$. 
Furthermore, each such $\Lambda$ serves as Frobenius form
for $H^*$. In this section, we will fix the unique $\Lambda \in \Hint_H$ such that
\begin{equation}
\label{E:Lambda}
\gen{\e,\Lambda} = \dim_\k H
\end{equation}
We also fix the following normalized version of the regular character $\chi_\reg$ of $H$; see \S\ref{SS:Trace}:
\begin{equation}
\label{E:lambda}
\lambda:= (\dim_\k H)^{-1} \chi_\reg 
\end{equation}
Then $\lambda$
is a nonzero integral of $H^*$ satisfying $\gen{\lambda,\Lambda} = \gen{\lambda,1} = 1$.
Taking $\lambda$ as Frobenius form for $H$, the associated Casimir element is 
\begin{equation}
\label{E:CasCh0}
\begin{aligned}
\Cas{\lambda} &= \ant(\Lambda_{(1)}) \otimes \Lambda_{(2)} = \Lambda_{(2)} \otimes \ant(\Lambda_{(1)}) \\
&= \ant(\Lambda_{(2)}) \otimes \Lambda_{(1)} = \Lambda_{(1)} \otimes \ant(\Lambda_{(2)})
\end{aligned}
\end{equation}
Thus, the Casimir trace $\Hig{\lambda} \colon H \to \cen H$ is 
given by  $\Hig{\lambda}(h) = \ant(\Lambda_{(1)}) h \Lambda_{(2)}$ for $h \in H$. Therefore,
\begin{equation}
\label{E:HigValue}
\Hig{\lambda}(1) = \gen{\e, \Lambda} = \dim_\k H
\end{equation}
Reversing the roles of $H$ and $H^*$ and taking $\Lambda$ as a Frobenius form for $H^*$,
the value of the Casimir trace $\Hig{\Lambda}$ at $\e = 1_{H^*}$ is
\begin{equation}
\label{E:HigValue.2}
\Hig{\Lambda}(\e) = \gen{\lambda, 1} = 1
\end{equation}

\begin{example}
For a finite group algebra $H = \k G$, the integral $\Lambda$ in \eqref{E:Lambda} 
is $\Lambda =  \sum_{g \in G} g$. The normalized regular
character $\lambda$ in \eqref{E:lambda} is given by $\gen{\lambda,g} = \d_{g,1}$ 
for $g\in G$; so $\lambda$ is identical to the Frobenius form of \S\ref{SSS:kG}. 
The Casimir element $\Cas\lambda$ in \eqref{E:CasCh0} is
$\sum_{g\in G} g \otimes g^{-1}$ as in \S\ref{SS:FrobFinitekG}. So $\Hig\lambda(h) = \sum_{g \in G} g h g^{-1}$
for $h \in \k G$.
\end{example}

\subsection{Frobenius Divisibility for Hopf Algebras}
\label{SS:FrobDivHopf}

The following special case of Corollary~\ref{C:Divisibility} is due to
Cuadra and Meir \cite[Theorem 3.4]{jCeMxx}. Note that \eqref{E:CasCh0} gives a formula for the
Casimir element to be tested for integrality. For $H = \k G$, the theorem gives 
Frobenius' original result (\S\ref{SS:FrobFinitekG}). 

\begin{thm}[Cuadra and Meir]
\label{T:FrobHopf}
Then the following are equivalent for 
a split semisimple Hopf algebra $H$ over a field $\k$ of characteristic $0$. 
\begin{enumerate}
\renewcommand{\labelenumi}{\normalfont (\roman{enumi})}
\item
Frobenius divisibility for $H$: the degrees of all irreducible representations of $H$ divide $\dim_\k H$;
\item
the Casimir element \eqref{E:CasCh0} is integral over $\ZZ$.
\end{enumerate}
\end{thm}

\begin{proof}
Choosing the Frobenius form $\lambda$ for $H$ as in \eqref{E:lambda}, the Casimir element
$\Cas{\lambda}$ is given by \eqref{E:CasCh0}
and $\Hig{\lambda}(1) = \dim_\k H$ by \eqref{E:HigValue}. Thus, the theorem is a
consequence of Corollary~\ref{C:Divisibility}.
\end{proof}

\subsection{More Preliminaries: The Representation Algebra and the Character Map}
\label{SS:prelim2}

We continue to let $H$ denote a split semisimple Hopf algebra over a field $\k$
of characteristic $0$. Our remaining applications of the material of Section~\ref{S:Symmetric} all involve
the \emph{representation ring} $\sR(H)$ of $H$. We remind the reader that $\sR(H)$, by definition, is the abelian group
with generators the isomorphism classes $[V]$ of representations $V \in \Rep H$ and with
relations $[V \oplus W] = [V]+[W]$ for $V,W \in \Rep H$. The multiplication of $\sR(H)$ comes from the
tensor product of representations: $[V][W] = [V \otimes W]$. As a group, $\sR(H)$ is free abelian 
of finite rank, with $\ZZ$-basis
given by the classes $[S]$ with $S \in \Irr H$; so all elements of $\sR(H)$ are integral over $\ZZ$. 

We shall also consider the $\k$-algebra $\sR_\k(H):= \sR(H) \otimes_{\ZZ} \k$, which 
can be thought of as a subalgebra of $H^*$ via the \emph{character map}
\[
\begin{tikzpicture}[baseline=(current  bounding  box.center),  
bij/.style={above,sloped,inner sep=0.5pt}, >=latex, scale=.7]
\matrix (m) [matrix of math nodes, row sep=.1em,
column sep=2em, text height=1.5ex, text depth=0.25ex]
{\chi_\k \colon \sR_\k(H) & H^* \\ \qquad \upin & \upin \\  \qquad [V] \otimes 1 & \chi_V \\};
\draw[right hook->] (m-1-1) edge  (m-1-2);
\draw[|->] (m-3-1) edge (m-3-2);
\end{tikzpicture} 
\]
The image of this map
is the algebra $(H/[H,H])^*$ of all trace forms on $H$ or, equivalently, the algebra of
all cocommutative elements of $H^*$.
We remind the reader of some standard facts about
$\sR_\k(H)$; for more details, see \cite{mL98} for example.
The algebra $\sR_\k(H)$ is finite-dimensional semisimple. 
A Frobenius form for $\sR_\k(H)$ is given by the dimension of $H$-invariants:
\[
\d \colon \sR_\k(H) \to \k , \qquad [V] \otimes 1 \mapsto (\dim_\k V^H)1_\k
\]
The Casimir element $\Cas\d$ is the image of the element $\sum_{S \in \Irr H} [S] \otimes [S^*] \in \sR(H)^{\otimes 2}$
in $\sR_\k(H)^{\otimes 2}$. Consequently, $\Cas\d$ is integral over $\ZZ$, because the ring $\sR(H)^{\otimes 2}$ is 
finitely generated as a $\ZZ$-module and so all its elements are integral over $\ZZ$.
Finally, the character map does in fact give an embedding of symmetric $\k$-algebras,
\begin{equation}
\label{E:CharMap}
\chi_\k = \chi \otimes \k \colon (\sR_\k(H),\d) \into (H^*,\Lambda_0)
\end{equation}
where $\Lambda_0 = (\dim_\k H)^{-1}\Lambda$
is the unique integral of $H$ satisfying $\gen{\e,\Lambda_0}= 1$.

\subsection{Characters that are Central in $H^*$}
\label{SS:characterCenterH*}

As an application of Proposition~\ref{P:cprid},
we offer the following elegant generalization of 
Frobenius' Divisibility Theorem
due to S. Zhu \cite[Theorem 8]{sZ93}.
Note that the hypothesis $\chi_S \in \cen(H^*)$ is of course automatic for 
finite group algebras $H = \k G$, because $H^*$ is commutative in this case. 

\begin{thm}[S. Zhu]
\label{T:Zhu}
Let $H$ be a semisimple Hopf algebra over a field $\k$
of characteristic $0$.
Then $\dim_\k S$ divides $\dim_\k H$ for every absolutely irreducible $S \in \Irr H$ satisfying $\chi_S \in \cen(H^*)$. 
\end{thm}

\begin{proof}
Since semisimple Hopf algebras are separable, we may assume that $\k$ is algebraically closed. 
Thus, $H$ and $H^*$ are
both split semisimple. Choose $\Lambda \in 
\Hint_H$ as in \eqref{E:Lambda}; so 
$\Lambda$ is the character of the regular representation of $H^*$. Then, with $\lambda$ as in
\eqref{E:lambda}, we have $\Hig{\lambda}(1) = \dim_\k H$ and $\Cas{\lambda}$
is given by \eqref{E:CasCh0}. 
Thus, Proposition~\ref{P:cprid} gives the following formula
for the primitive central idempotent $\cprid{S} \in \cen H$:
\[
\cprid{S}\,\tfrac{\dim_{\k}H}{\dim_\k S} \underset{\text{Prop.~\ref{P:cprid}}}{=}   
\chi_S(\ant(\Lambda_{(1)}))\Lambda_{(2)} =  \Lambda \lhk \ant^*(\chi_S) = \Lambda \lhk \chi_{S^*}
\]
It suffices to show that the element $\Lambda \lhk \chi_{S^*}\in \cen H$
is integral over $\ZZ$. 
First, note that $\chi_{S^*}$ is integral over $\ZZ$, because this holds for
$[S^*] \in \sR(H)$.
Furthermore, by hypothesis, $\chi_{S^*} = \ant^*(\chi_S) \in \cen(H^*)$ and so $\chi_{S^*}$ belongs to the
integral closure $\cen(H^*)^{\text{int}}:= \{ f \in \cen(H^*) \mid f \text{ is integral over } \ZZ \}$.
Thus, it suffices to show that all elements of $\Lambda \lhk \cen(H^*)^{\text{int}}$ are integral over $\ZZ$\,. 
But $\cen(H^*) = \sum_{M \in \Irr H^*} \k \cprid{M} \cong \k \times \k \times \dots \times \k$ and so
$\cen(H^*)^{\text{int}} = \sum_{M \in \Irr H^*} \cO \cprid{M}$, where $\cO$
denotes the integral closure of $\ZZ$ in $\k$.
Furthermore, $\Lambda \lhk \cprid{M} = (\dim_\k M) \chi_M$ by \eqref{E:cpridSymm}.
Therefore, $\Lambda \lhk \cen(H^*)^{\text{int}}  \subseteq \chi(\sR(H^*))\cO$. 
Finally, since the ring $\chi(\sR(H^*))\cO$ is a finitely generated $\cO$-module, all its elements are
integral over $\ZZ$, completing the proof.
\end{proof}

\subsection{The Class Equation}
\label{SS:ClassEqn}

We now prove the celebrated class equation due to Kac \cite[Theorem 2]{gK72}
and Y. Zhu \cite[Theorem 1]{yZ94}; the proof given here is based on \cite{mL98}.
Recall that the representation algebra $\sR_\k(H)$ embeds into $H^*$ via the character map \eqref{E:CharMap}.
Thus, for any $M$ in $\Rep \sR_\k(H)$, we may
consider the induced module $\Ind_{\sR_\k(H)}^{H^*} M$. 

\begin{thm}[Kac, Y. Zhu]
\label{T:ClassEqn}
Let $H$ be a semisimple Hopf algebra over an algebraically closed field $\k$
of characteristic $0$. Then
$\dim_\k \Ind^{H^*}_{\sR_\k(H)}M$ divides $\dim_\k H$
for every $M$ in $\Irr \sR_\k(H)$.
\end{thm}

\begin{proof}
This is an application of Corollary~\ref{C:Divisibility.2} to the the character map \eqref{E:CharMap}.
The main hypotheses have been checked in \S\ref{SS:prelim2}, in particular integrality of 
the Casimir element $\Cas{\d}$ over $\ZZ$.
In addition, note that, $\Hig{\Lambda_0}(\e) = \dim_\k H$ by \eqref{E:HigValue.2}. 
Therefore, Corollary~\ref{C:Divisibility.2} applies and yields 
that the fraction 
$\frac{\dim_\k H}{\dim_\k \Ind^{H^*}_{\sR_\k(H)}M}$
is integral over $\ZZ$, proving the theorem.
\end{proof}

Frobenius' Divisibility Theorem for finite group algebras $\k G$ also
follows from Theorem~\ref{T:ClassEqn} applied to $H = (\k G)^*$, because 
$\chi_\k \colon \sR_\k(H) \longiso H^* = \k G$ in this case.

\subsection{Factorizable Hopf Algebras}
\label{SS:factorizable}

We remind the reader of some facts about factorizable Hopf algebras. 
Let $H$ be a Hopf algebra, which need not be 
finite-dimensional for now. Following Drinfeld \cite{vD89}, $H$ is called \emph{almost cocommutative} if there is an 
$R\in \Units{(H\otimes H)}$ satisfying the condition 
\begin{equation*}
\tau(\Delta(h))R = R\Delta(h) \qquad (h \in H)
\end{equation*}
where $\tau \in \Aut_{\kalg}(H\otimes H)$ is the switch map as 
in Lemma~\ref{L:Casimir}.  
An almost cocommutative Hopf algebra $(H,R)$ is called \emph{quasitriangular} if:
\begin{equation*}
\begin{aligned}
&\Delta(R^1)\otimes R^2=R^1\otimes r^1\otimes R^2r^2 \\
&R^1\otimes\Delta(R^2)=R^1r^1\otimes r^2\otimes R^2
\end{aligned}
\end{equation*}
Here, elements $t \in H \otimes H$ are symbolically written as $t = t^1 \otimes t^2$, with summation 
over the superscript being assumed, and we have written
$R = R^1\otimes R^2 = r^1 \otimes r^2$ to indicate two different summation indices.
Put $b:=\tau(R)R = r^2R^1 \otimes r^1R^2 \in H \otimes H$ and 
define a $\k$-linear map $\Phi = \Phi_R$ by
\begin{equation*}
\begin{tikzpicture}[baseline=(current  bounding  box.center),  >=latex, scale=.7,
bij/.style={above,sloped,inner sep=0.5pt}]
\matrix (m) [matrix of math nodes, row sep=.1em,
column sep=2.5em, text height=1.5ex, text depth=0.25ex]
{\Phi  \colon &[-3em] H^* & H  \\ 
& \upin & \upin \\ 
& f & b^1\gen{f,b^2}\\}; 
\draw[->] 
(m-1-2) edge  (m-1-3);
\draw[|->] 
(m-3-2) edge (m-3-3);
\end{tikzpicture} 
\end{equation*}
The quasitriangular Hopf algebra $(H,R)$ is called \emph{factorizable} if $\Phi$ is bijective. 
Note that this forces $H$ to be finite-dimensional.
An important example of a factorizable Hopf algebra is the Drinfeld double of any finite-dimensional 
Hopf algebra; see \cite[Theorem 13.2.1]{dR12}. 

Now let us again focus on the case where $H$ is a semisimple Hopf 
algebra over an algebraically closed field $\k$ with $\ch \k = 0$.
Consider the representation algebra $\sR_\k(H)$ and the embedding of
symmetric algebras $(\sR_\k(H),\d) \into (H^*,\Lambda_0)$ given by the character map \eqref{E:CharMap}
and recall that the image of this map 
is the algebra $(H/[H,H])^*$ of all trace forms on $H$. Recall also from \eqref{E:lambda} that
$\lambda = (\dim_\k H)^{-1} \chi_\reg \in (H/[H,H])^*$ is a nonzero integral of $H^*$.


\begin{prop}
\label{F:Prop}
Let $(H,R)$ be a factorizable Hopf algebra over an algebraically closed field $\k$ with $\ch \k = 0$
and assume that $H$ is a semisimple. Then the map
$\Psi = \Phi \circ \chi_\k \colon (\sR_\k(H),\delta)\into (H,\lambda)$ is an embedding of symmetric algebras with
image $\Im \Psi = \cen H$.
\end{prop}

\begin{proof}
By \cite[Theorem 2.3]{hjS01}, the restriction of $\Phi$ to $(H/[H,H])^*$ is an algebra isomorphism with $\cen H$.
So we just need to check that $\lambda\circ \Psi =\delta$ or, equivalently,
$\gen{\lambda, \Phi(c)} = \gen{c,\Lambda_0}$
for all $c \in (H/[H,H])^*$. Since $\lambda$ and $c$ are trace forms, we compute
\[
\begin{aligned}
\gen{\lambda, \Phi(c)} &= \gen{\lambda, b^1\gen{c,b^2}} 
= \gen{\lambda, b^1}\gen{c,b^2} 
= \gen{\lambda, r^2R^1}\gen{c,r^1R^2} 
= \gen{\lambda, R^1r^2}\gen{c, R^2r^1} \\
&= \gen{c, b^1}\gen{\lambda, b^2}
= \gen{ c, b^1\gen{\lambda,b^2}} = \gen{c, \Phi(\lambda)}
\end{aligned}
\]
Thus, it suffices to show that $\Phi(\lambda) = \Lambda_0$, where $\Lambda_0 \in H$ is as in \S\ref{SS:prelim2}; so
$\Lambda_0$ is the unique integral of $H$ satisfying $\gen{\e,\Lambda_0}= 1$.

But $\gen{\e,b^1}b^2 = \gen{\e,r^2}\gen{\e,R^1}r^1R^2 = 1$, because 
$R^1\gen{\e,R^2} = 1 = \gen{\e,R^1}R^2 $ by \cite[Proposition 10.1.8]{sM93}. 
Hence, for any $f \in H^*$,
\[
\gen{\e,\Phi(f)} = \gen{\e, b^1\gen{f,b^2}} = \gen{\e, b^1}\gen{f,b^2} = \gen{f,  \gen{\e,b^1}b^2} = 
\gen{f,1}
\]
Using this and the identity $\Phi(fc) = \Phi(f)\Phi(c)$ for $f \in H^*$ and $c \in (H/[H,H])^*$ from \cite[Theorem 2.1]{hjS01},
we obtain, for any $h \in H$,
\[
h\Phi(\lambda) = \Phi( \Phi^{-1}(h)\lambda) = \Phi(\gen{\Phi^{-1}(h), 1}\lambda) = 
\gen{\Phi^{-1}(h), 1} \Phi(\lambda) = \gen{\e,h} \Phi(\lambda)
\]
Thus, $\Phi(\lambda)$ is an integral of $H$ and, since $\gen{\e,\Phi(\lambda)} = \gen{\lambda,1} = 1$,
we must have $\Phi(\lambda)=\Lambda_0$ as desired.
\end{proof}

We are now ready to prove the following result of Schneider \cite[Theorem 3.2]{hjS01}.

\begin{thm}[Schneider]
\label{T:Factorizable}
Let $H$ be a semisimple factorizable Hopf algebra over an algebraically closed field $\k$
of characteristic $0$. Then
$(\dim_\k S)^2$ divides $\dim_\k H$
for every $S \in \Irr H$.
\end{thm}

\begin{proof}
Consider the primitive central idempotent $\cprid{S} \in \cen H$. 
The preimage $\Psi^{-1}(\cprid{S})$ is a primitive idempotent in the semisimple algebra $\sR_\k(H)$;
so $S':= \sR_\k(H) \Psi^{-1}(\cprid{S})$ is an irreducible representation of $\sR_\k(H)$, with
$\Ind_{\sR_\k(H)}^H(S') \cong H\cprid{S} \cong \End_\k(S)$.
By Proposition \ref{F:Prop} and our remarks in \S\ref{SS:prelim2},
we can apply Corollary \ref{C:Divisibility.2} to the map $\Psi$ to get 
that $\dim_\k \Ind_{\sR_\k(H)}^H(S') = (\dim_\k S)^2$ divides $\Hig{\lambda}(1) = \dim_\k H$.
\end{proof}

We mention in closing that Schneider's Theorem implies an earlier result of Etingof and Gelaki
\cite{pEsG98a}, which states that Kaplansky's conjecture holds for any
quasitriangular semisimple Hopf algebra $H$ over an algebraically closed field 
of characteristic $0$. Indeed, the Drinfeld double $D(H)$ is a semisimple
factorizable Hopf algebra of dimension $(\dim_\k H)^2$ \cite[Corollary 13.2.3]{dR12} which maps onto $H$ \cite{pEsG98a}.
Therefore, any $S \in \Irr H$ can be viewed as an irreducible representation of $D(H)$.
Theorem~\ref{T:Factorizable} gives that $(\dim_\k S)^2$ divides $\dim_\k D(H)$; so
$\dim_\k S$ divides $\dim_\k H$.

\begin{ack}
The main results of this article were presented by the junior author in a lecture during the
\textit{International Conference on Groups, Rings, Group Rings and Hopf Algebras}
(celebrating the 75th birthday of Donald S. Passman)
at Loyola University, Chicago, October 2-4, 2015.
Both authors would like to thank the organizers for the invitation and for their hospitality.
\end{ack}

\renewcommand{\emph}[1]{{\itshape #1}}

\bibliographystyle{amsplain}

\def\cprime{$'$} 
\providecommand{\bysame}{\leavevmode\hbox to3em{\hrulefill}\thinspace}
\providecommand{\MR}{\relax\ifhmode\unskip\space\fi MR }
\providecommand{\MRhref}[2]{%
  \href{http://www.ams.org/mathscinet-getitem?mr=#1}{#2}
}
\providecommand{\href}[2]{#2}


\end{document}